\documentclass[12pt]{article}
\usepackage{amssymb,amsmath}
\usepackage{hyperref}
\usepackage{xcolor}
\hypersetup{
	colorlinks,
	linkcolor={red!60!black},
	citecolor={green!60!black}
}

\newcommand{\qed}{\hskip 5mm \rule{2.5mm}{2.5mm}\vskip 10pt}
\newcommand{\R}{{\mathbb R}}

\newcommand{\N}{{\mathbb N}}

\newcommand{\E}{{\mathbb E}}
\renewcommand{\P}{{\mathbb P}}

\numberwithin{equation}{section}

\begin{document}
\newtheorem{theorem}{Theorem}[section]
\newtheorem{definition}[theorem]{Definition}
\newtheorem{lemma}[theorem]{Lemma}
\newtheorem{note}[theorem]{Note}
\newtheorem{corollary}[theorem]{Corollary}
\newtheorem{proposition}[theorem]{Proposition}
\newtheorem{remark}[theorem]{Remark}
\renewcommand{\theequation}{\arabic{section}.\arabic{equation}}
\newcommand{\newsection}[1]{\setcounter{equation}{0} \section{#1}}
\title{The Stein-Chen method and a Law of Small Numbers in Riesz Spaces}
\author{
 Wen-Chi Kuo${}^\sharp {}^\flat$\\
 Nigel Musara${}^\sharp$\\
 Bruce A Watson${}^\sharp {}^\circ$\\
  \\
 ${}^\sharp$ School of Mathematics,
 University of the Witwatersrand\\
 Private Bag 3, P O WITS 2050, South Africa\\
 \\
 ${}^\flat$ NITheCS.\\
${}^\circ$ DSTI-NRF Centre of Excellence in Mathematical and Statistical Sciences.
  }
\maketitle
\abstract{\noindent Martingales, Markov processes and Laws of Large Numbers have been well studied in the Riesz space (vector lattice) setting. There has, however, been no attention given  in the Riesz space setting to Laws of Small Numbers or to the so called Stein-Chen method. Here we adapt the Stein-Chen method to the Riesz space setting and hence give a conditional Laws of Small Numbers for Bernoulli processes in Riesz spaces. This requires extensive use of functional calculus and the associated $f$-algebra structure.
}

\parindent=0cm
\parskip=0.5cm
\section{Introduction} \label{s: introduction}

The study of random processes in Riesz spaces began with the generalization of 
martingale theory, see DeMarr \cite{DeMarr}, Stoica \cite{Stoica},  Kuo, Labuschagne and Watson \cite{KLW1} and Grobler \cite{jensen}. 
A law of large numbers was given via ergodic theory in \cite{ergodic} and for mixingales in Riesz spaces in \cite{KVW1}. 
More concrete processes such as Markov processes, 
Bernoulli processes and auto regressive processes of order 1 were considered in the Riesz space setting in \cite{markov, KRW2, KVW2, vw-markov}. 
The current work
follows the trends of the above three cited works by considering Laws of Small Numbers in Riesz spaces via an extension of the Stein-Chen method to Riesz spaces. 
Here, we apply the Stein-Chen method to a conditionally independent sequences of components of a weak order unit to obtain a rate of convergence to a Poisson distribution and hence a Law of Small Numbers for such a process. 
This is one of the simplest processes in the Riesz space setting and gives a model problem which highlights the hurdles that need to be overcome when dealing with Laws of Small Numbers for more general processes in Riesz spaces. 
An up to date survey of the Laws of Small Numbers can be found in \cite{FHR}, while interest in Laws of Small Numbers dates at as far back as 1890, see  \cite[page 118]{Haight}.

The roots of the so called Stein-Chen method for Poisson approximation lie in the recurrence formula (2.12) of \cite{Stein-1970} and (2.3) of \cite{Chen}.  
Since the appearance of \cite{Chen, Stein-1970}, this approach has been streamlined and applied to more and more general processes to yield Laws of Small Numbers for them. 
A brief distillation of the Stein-Chen method, by Barbour and Chen, can be found in the Preface of \cite{BC}, while a comprehensive coverage of the ideas behind the method and some of its applications are given in \cite{CGR} and the lectures of Stein \cite{Stein}. 
There is however a vast literature on the Stein-Chen method and its applications, see \cite{BH, CGR, CGS, FHR, Haight, Ross} and their bibliographies. 

In particular, if $X_1,\dots,X_n$ is a sequence of $0-1$ valued independent random variables in a probability space  $(\Omega,\mathcal{A},\mathbb{P})$,  Stein, see \cite{Stein}, showed that 
$\displaystyle{\max_{j=1,\dots,n} p_j}$, where $p_i=\mathbb{E}[X_i=1]$,
is an upper bound on the difference between the distribution of $\displaystyle{W=\sum_{i=1}^n X_i}$ and the Poisson distribution with parameter $\displaystyle{\lambda=\sum_{i=1}^n p_i}$.
Thus the Poisson distribution with parameter $\lambda$ is a good approximation to the distribution of $W$ if the $p_i, i=1,\dots,n,$ are small and hence the term `laws of small numbers'.This type of result is closely linked to the Chernoff inequality given by Ben Amor and Omrani in \cite{Amine-Amal}.
This project was initiated as a research project, \cite{NM},  by Musara while a postgraduate student at the University of the Witwatersrsand the under the supervision of Kuo and Watson. 
We refer the readers to \cite{NM} for further background on the Stein-Chen method.

In Section 2, we give the aspects of Riesz spaces and $f$-algebras used in this paper. In Section 3,  we give the lifting of the Poisson distribution to Riesz spaces, in Section 4 we generalize the approach of Stein and Chen to Riesz spaces to obtain two laws of small number. In Section 5 we present an application.


\section{Riesz space preliminaries} \label{preliminary section}

All material from this section is known and can be found in \cite{expectation, IFSA,ergodic, KRW,vw-markov}, it is placed here for the reader's convenience.
We say that a linear operator $T$  on an Archimedean 
Riesz space (vector lattice), $E$, with weak order unit, is a conditional expectation operator  if $T$ is a positive order continuous projection which maps weak order units to weak order units and has $R(T)$ Dedekind complete. 
If, in addition, $T|x|=0$ for $x\in E$ implies $x=0$, then we say that $T$ is strictly positive.
For $T$ a conditional expectation operator on $E$, since $T$ is a projection and $T$ maps weak order units to weak order units, there is a weak order unit, say $u$, with $u=Tu$.

We denote the positive cone of $E$ by $E_+=\{x\in E\,|\, x\ge 0\}$. If 
$w\in E_+$, we say that $v$ is a component of $w$ if $0\le v\le w$ and $(w-v)\wedge v=0$.
If $f\in E_+$, we denote the band projection onto the band generated by $f$ by $P_f$ and the associated component of $u$ by $u_f:=P_fu$.

Let $E$ be a Dedekind complete Riesz space with weak order unit and $T$ be a strictly positive conditional expectation operator on $E$.
The $T$-universal completion of $E$ is 
 $$\hat{E}=\mbox{dom}(T)-\mbox{dom}(T):=\{f-g|f,g\in\mbox{dom}(T)\},$$
where
$$\mbox{dom}(T):=\{f\in E^u_+|\exists  \mbox{ net } f_\alpha\uparrow f \mbox{ in } E^u,
(f_\alpha)\subset E_+, Tf_\alpha \mbox{ bdd in } E^u\}$$
and $E^u$ is the universal completion of $E$.
$\hat{E}$ is a Dedekind complete Riesz space containing $E$ as an order dense subspace, and each weak order unit of $E$ is again a weak order unit of $\hat{E}$.
The space $\hat{E}$ will be denoted $L^1(T)$, see \cite{expectation}.
Further,  $T$ admits unique extension to a conditional expectation on the $T$-universal completion, $\hat{E}$, of $E$.
In particular, if $f\in \hat{E}_+$ then there is a net $(f_\alpha)$ in $E_+$ with $(Tf_\alpha)$ bounded in
 the universal completion, $E^u$, of $E$, with $f_\alpha\uparrow f$.
Then $Tf_\alpha\uparrow$ in $E$ and is order bounded in $E^u$ and  as such has limit, which we denote $\hat{T}(f)$,
  in $\hat{E}$.
We extend $\hat{T}$ to the whole of $\hat{E}$ by setting 
 $$\hat{T}f=\hat{T}f^+-\hat{T}f^-,\quad f\in\hat{E}.$$
It can now be verified that the extension $\hat{T}:\hat{E}\to \hat{E}$ of $T$ is a 
 conditional expectation operator on $\hat{E}$.  The details can be found in \cite{expectation}.
The essential ideas for this extension of the space and operator originated in a paper of Grobler and de Pagter \cite{GdP}.

 Let $E$ be a $T$-universally complete Riesz space, where $T$ is a conditional expectation
  operator on 
  $E$, and let $u$ be a weak order unit for $E$ with $Tu=u$.
  $R(T)$ is universally complete and hence an $f$-algebra, see \cite{KRW}.
  If $u\in R(T)$ is a weak order unit, then $u$ is invertible in $R(T)$, and hence in $E$. 
  See the Appendix for more details on inverses and partial inverses.
  Furthermore, $L^1(T)$ is an $R(T)$-module and 
   $T$ is an averaging operator in the sense that if $f\in R(T)$ 
and $g\in E=L^1(T)$ then $fg\in E$ with $T(fg)=fT(g)$.
It should be noted that if $u=Tu$ is the weak order unit of $E$ chosen to be the algebraic unit
of $R(T)$ and $E^u$, then for components $p$ and $q$ of $u$ we have that $p\cdot q=p\wedge q$. We also note that if $p$ is a component of $u$ then $p\cdot f=P_p f$ where 
$P_p$ is the band projection on the band generated by $p$. Further if $f\in E_+$ then 
$p\cdot f=p\wedge f$.
There is a bijective correspondence, ${\frak F}$, between the band projections ${\cal{BP}}$ on $L^1(T)$ and the components $C_u$ of $u$ in $L^1(T)$, given by 
 ${\frak F}[P]=Pu$.  Here $$Pf=(Pu)\cdot f= {\frak F}[P]\cdot f,$$
for all $f\in L^1(T)$.

The concept of $T$-conditional independence was generalized from the probability space setting to that of a Dedekind complete Riesz space, say $E$, with weak order unit, say $u$, and conditional expectation in $T$ having $Tu=u$ as follows in
\cite[Definition 4.1]{ergodic}.

\begin{definition}
 Let $E$ be a Dedekind complete Riesz space with conditional expectation $T$ and weak order unit $u=Tu$.
 Let $P$ and $Q$ be band projections on $E$, we say that $P$ and $Q$ are $T$-conditionally independent if
 \begin{eqnarray}
  TPTQu=TPQu=TQTPu.\label{indep-e}
 \end{eqnarray}
 We say that two Riesz subspaces $E_1$ and $E_2$ of $E$ are $T$-conditionally independent if
 all band projections $P_i, i=1,2,$ in $E$ with  $P_ie\in E_i, i=1,2,$ are $T$-conditionally independent.
\end{definition}

The above definition, in the case of the $R(T)$-module $L^1(T)$ can be expressed as follows.

\begin{lemma}\label{lem-ind-comp-1}
 Let $p_1$ and $p_2$ be components of the weak order unit $u=Tu$ in $E=L^1(T)$,
 then $p_1$ and $p_2$ 
 are $T$-conditionally independent if and only if
 \begin{eqnarray}
  Tp_1\cdot Tp_2=T(p_1p_2).\label{indep-comp-u}
 \end{eqnarray}
 Two Riesz subspaces $E_1$ and $E_2$ of $E$ are $T$-conditionally independent if and only if
 each pair of components $p_i\in E_i, i=1,2,$ of $u$ are $T$-conditionally independent.
\end{lemma}

\begin{proof}
Since $T$ is an averaging operator,  for all components $p_1$ and $p_2$ of $u$,
 $$T(p_1\cdot Tp_2)=(Tp_1)\cdot (Tp_2)=T(p_2\cdot Tp_1).$$ 
The above noted correspondence between band projections and components of the chosen weak order unit gives the remainder of the result.
\qed
\end{proof}

The following corollary, from \cite{vw-markov},  written in terms of components of $u$, 
relates $T$-conditional independence of the components $p$ and $q$ of $u$
 with
$T$-conditional independence of the closed Riesz subspaces $\left< p, {R}(T)\right>$ and $\left< q, {R}(T)\right>$
generated by $p$ and ${R}(T)$ and by $q$ and ${R}(T)$ respectively.

\begin{corollary}
 Let $p_i, i=1,2,$ be components of the weak order unit $u=Tu$ in $L^1(T)$.
 Then $p_i, i=1,2,$  are $T$-conditionally independent
 if and only if the closed Riesz subspaces
 $E_i=\left< p_i, {R}(T)\right>, i=1,2,$ are $T$-conditionally independent.
\end{corollary}

For ease of notation, if $(E_\lambda)_{\lambda\in \Lambda}$ is a family of Riesz subspaces of $E$ we denote the closed Riesz subspace of $E$, generated by $(E_\lambda)_{\lambda\in \Lambda}$, by
$$E_{\Lambda} := \left< \bigcup_{\lambda\in\Lambda_j} E_\lambda \right>.$$

\begin{definition}\label{def-n}
 Let $E_\lambda, \lambda\in \Lambda,$ be a  family of closed Riesz subspaces of $L^1(T)$ having
 ${R}(T)\subset E_\lambda$ for all $\lambda\in\Lambda$. We say that the family is $T$-conditionally independent
 if, for each pair of disjoint sets $\Lambda_1, \Lambda_2 \subset \Lambda$, we have that the pair
 $E_{\Lambda_1}$ and $E_{\Lambda_2}$ is $T$-conditionally independent.
\end{definition}

If $F$ is a universally complete Riesz space with weak order unit, say $e$, then $F$ is an $f$-algebra and $e$ can be taken as the algebraic unit, see \cite[Theorem 3.6]{venter}.
In this setting, 
 we say that $g\in F$ has a partial inverse if there exists $h\in F$ such that $gh=hg=P_{|g|}e$ where $P_{|g|}$ denotes the band projection onto the band generated by $|g|$. 
We refer to $h$ as the canonical partial inverse of $g$ if, in addition, to being a partial inverse to $g$, we have that $(I-P_{|g|})h=0$, i.e.,  $h\in {\cal B}_{|g|}$,  where ${\cal B}_{|g|}$ is the band generated by $|g|$.
Here ${\cal B}_f$ denotes the band generated by $f$ and $P_f$ the band projection onto ${\cal B}_f$.

The existence, uniqueness and positivity results concerning partial inverses and canonical partial inverses can be deduced from \cite[Theorem 5]{HdP-1986} and \cite[Remark 3.3]{RS-2019}, we summarize them in the below theorem.

 \begin{theorem}
Let $F$ be a  universally complete Riesz space with weak order unit, say $e$, which we also take as the algebraic unit of  the associated $f$-algebra structure.  Each $g\in F$ has a partial inverse $h\in F$. 
The canonical partial inverse of $g$ exists, is unique and in this case $g$ is also the canonical partial inverse of $h$. If $g\in F_+$ then so is its canonical partial inverse.
  \end{theorem}
 
The following aspect of the averaging property of conditional expectation operators was first proved in Riesz spaces in \cite[Corollary 2.3]{IFSA}.

\begin{lemma}\label{lem-bands}
Let $E$ be a Dedekind complete Riesz space with weak order unit $u$ and strictly positive conditional expectation operator $T$ with $Tu=u$. 
If $f\in E_+$ then the band projection $P_f$ onto the band generated by $f$ and the band projection onto the band generated by $Tf$ are related by $P_f\le P_{Tf}$.   
 \end{lemma}

\section{Poisson distribution in Riesz spaces}

We extend the definition of  the Poisson distribution with parameter $\lambda$, denoted 
${\rm Po}(\cdot;\lambda)$, to the  setting of a Dedekind complete Riesz space $E$ with 
weak order unit, say $u$, and conditional expectation $T$ with $Tu=u$. 
For this,  we assume that $E$ is $T$-universally complete,  as in this case $R(T)$ is a universally complete $f$-algebra with algebraic unit $u$.
Thus for each $H\in R(T)$,   $H^k\in R(T)$ for all
$k\in \mathbb{N}$.  For brevity of notation, we set $H^0:=u$. 

Recall that we denote $u_g:=P_gu$ where $P_g$ is the band projection on the band, $B_g$,  in $E$, generated by $g\in E_+$.

For each $t\in \R$ and $H\in R(T)_+$,
 we have 
 $(tu-H)^+\in R(T),$ so,
 by \cite[Theorem 3.2]{expectation}, 
$u_{(tu-H)^+}\in R(T)_+$.
However, $e^{-H}$ can be expressed, via functional calculus and Freudenthal's theorem,  as 
$$e^{-H}=\sup_{n\in\N}\sum_{k=0}^{n2^n}e^{(k+1)2^{-n}}
P_{((k+1)2^{-n}-H)^+}(I-P_{(H-k2^{-n})^+})u.$$
Thus $e^{-H}\in R(T)_+$.
It is easily verified, in addition, that
$$e^{-H}=\sum_{k=0}^{\infty} \frac{(-H)^k}{k!}$$
where this summation is taken as the order limit of the partial sums.

We define 
\begin{equation}\label{12-28-poisson-1}
{\rm Po}(k;H):=\frac{H^ke^{-H}}{k!},
\end{equation}
for $k\in\N_0$.
This extends to an $R(T)$ valued measure on the $\sigma$-algebra, ${\cal P}(\N_0)$, of all subset of $\N_0$, by setting
\begin{equation}\label{2024-1-12-1}
{\rm Po}(A;H)=\sum_{k\in A}{\rm Po}(k;H),
\end{equation}
for each $A\subset \N_0$ and $H\in R(T)$.

\begin{definition}\label{def-poisson}
Let $E$ be a Dedekind complete Riesz space having weak order unit $u$ and conditional expectation operator $T$,  with $Tu=u$ and $E$ being $T$-universally complete.
Let $f\in E$ be a finite sum of components of $u$.  We say that  $f$ is 
 $T$-conditionally Poisson distributed with parameter $H$ in $R(T)_+$ if 
\begin{equation}\label{f-poisson-def}
T(I-P_{|f-ku|})u={\rm Po}(k;H),
\end{equation}
for each $k\in\N_0$.
\end{definition}

\begin{lemma}\label{pieces}
Let $E$ be a Dedekind complete Riesz space having weak order unit $u$ and conditional expectation operator $T$,  with $Tu=u$ and $E$ being $T$-universally complete.
Let $f\in E$ be a finite sum of components of $u$, then $f$
can be represented as
\begin{equation}\label{discrete-rep}
f=\sum_{j=0}^\kappa j r_j,
\end{equation}
 where each $r_j$ is a component of $u$, $j=0,\dots,\kappa$ with $r_ir_j=0$ for all $i\ne j$ and
$$\sum_{j=0}^\kappa r_j=u.$$
If $f$ obeys a $T$-conditional Poisson distribution, then (\ref{f-poisson-def}) becomes
\begin{equation}\label{f-expl-poisson}
Tr_j={\rm Po}(j;Tf)=\frac{(Tf)^j e^{-Tf}}{j!},
\end{equation}
for each $j=0,1,\dots,\kappa$.
Here $\kappa\in \N_0\cup\{\infty\}$.
\end{lemma}

\begin{proof}
 By direct calculation, from (\ref{discrete-rep}) we have that, for $k=0,\dots,\kappa$,
 $$u_{|f-ku|}=\sum_{j\ne k} r_j.$$
 From the above expression, we have
 $$u-u_{|f-ku|}=r_k.$$
So  $Tr_k=T(u-u_{|f-ku|})u$ and the remainder of the lemma follows from  (\ref{f-poisson-def}) and the assumption that  $f$   
obeys a $T$-conditional Poisson distribution.
\qed
\end{proof}

\section{Stein-Chen method in Riesz spaces}

Let $E$ be a Dedekind complete Riesz space with conditional expectation operator $T$ and weak order unit $u=Tu$. 
Let  
$q_1,\dots,q_n$ be $T$-conditionally independent components of $u$ in $E$
and
 $$w:=\sum_{i=1}^n q_i.$$
Define $$h_j:=Tq_j$$ and $$H:=\sum_{j=1}^n h_j=Tw.$$

The core question of this work is, how close to being Poisson distributed with parameter $H=Tw$ is $w$?
Recall that ${\rm Po}(k;H)\in R(T)$ for all $k\in \mathbb{N}_0$.
To answer this question we generalizing the recurrence process of Stein and Chen to conditional processes in Riesz spaces.
 
Let $J$ be the canonical partial inverse of $H$ in $R(T)$. Here
$J\cdot H=H\cdot J=P_Hu=u_H$ and $(I-P_H)J=0$. 
Here we recall that $u_H$ is the component of $u$ given by the application of the band projection generated by $H$ onto $u$.

Let $A\subset \N_0$.  Let $\nu_A(j)=\left\{\begin{array}{ll}u,& j\in A\\ 0, & j\not\in A\end{array}\right.$.
Define
 $g(0,H,A):=0$
and, for $j\ge 1$, 
\begin{equation}
g(j,H,A)=J\cdot[(j-1)g(j-1,H,A)+\nu_A(j-1)-{\rm Po}(A;H)]
+\frac{1}{j}(u-u_H)(\nu_A(0)- \nu_A(j)).
\label{2024-1-1}
\end{equation}

From (\ref{12-28-poisson-1}), we have
$$(u-u_H){\rm Po}(k;H)=(u-u_H)\frac{H^ke^{-H}}{k!}=0,$$
for all $k\in\N$, while, from the definition of $Po(0,H)$,
$$(u-u_H){\rm Po}(0;H)= (u-u_H)e^{-H}=u-u_H.$$
Thus
\begin{eqnarray}\label{dagger}
(u-u_H) {\rm Po}(A;H)=\sum_{k\in A}(u-u_H){\rm Po}(k;H) =(u-u_H)\nu_A(0).
\end{eqnarray}
Multiplying (\ref{2024-1-1}) by $u-u_H$ and substituting in (\ref{dagger}) we obtain
\begin{eqnarray}
(u-u_H)g(j,H,A)=\frac{1}{j}(u-u_H)(\nu_A(0)- \nu_A(j))
=\frac{1}{j}(u-u_H)( {\rm Po}(A;H)- \nu_A(j)),\label{2024-1-3} 
\end{eqnarray}
for $j\ge 1$.

For convenience we denote $N(n)=\{0,\dots, n\}$ for $n\in \N_0$. 

\begin{lemma}\label{lem-2024-1}
For $n\ge 1$,
\begin{equation}\label{riesz-recurrence-7-1}
g(n,H,A)=(n-1)!J^ne^H F(n-1,H,A)+\frac{1}{n}(u-u_H)({\rm Po}(A;H)-\nu_A(n)),
\end{equation}
where 
$$F(n,H,A)={\rm Po}(A\cap N(n);H)-{\rm Po}(A;H)\, {\rm Po}(N(n);H).$$
\end{lemma}

\begin{proof}As $E$ is an $R(T)$-module, we can
multiplying (\ref{2024-1-1}) by $H^j/(j-1)!$ to get, for $j\ge 2$,
\begin{eqnarray}\label{riesz-recurrence}
\frac{H^j}{(j-1)!}  g(j,H,A)=\frac{H^{j-1}}{(j-2)!} g(j-1,H,A)+\frac{H^{j-1}}{(j-1)!}[\nu_A(j-1) -{\rm Po}(A;H)].
\end{eqnarray}
We recall that $g(0,H,A):=0$ and hence, from (\ref{2024-1-1}),
\begin{equation}\label{riesz-recurrence-2-1*}
g(1,H,A)=J[\nu_A(0) - {\rm Po}(A;H)]+(u-u_H)[\nu_A(0)-\nu_A(1)],
\end{equation}
giving
\begin{equation}\label{riesz-recurrence-2-1**}
Hg(1,H,A)=u_H[\nu_A(0) - {\rm Po}(A;H)].
\end{equation}
So, for $k\ge 2$, summing (\ref{riesz-recurrence}) for $j=2,\dots,k$ gives
\begin{eqnarray}\label{riesz-recurrence-3-1*}
 \frac{H^k}{(k-1)!}  g(k,H,A)-H g(1,H,A)=\sum_{j=2}^k\frac{H^{j-1}}{(j-1)!}[\nu_A(j-1) -{\rm Po}(A;H)].
\end{eqnarray}
Thus 
\begin{eqnarray}\label{riesz-recurrence-4-1}
 \frac{H^k}{(k-1)!}  g(k,H,A)=e^H F(k-1,H,A) -(u-u_H)[\nu_A(0) -{\rm Po}(A;H)].
\end{eqnarray}
So multiplying (\ref{riesz-recurrence-4-1}) by $(k-1)!J^k$ and adding (\ref{2024-1-3}) gives the result of the lemma for $k\ge 2$, while for $k=1$ the result follows from (\ref{riesz-recurrence-2-1*}),  (\ref{riesz-recurrence-2-1*}) and the observations that
$\nu_A(0)=e^H{\rm Po}(N(0)\cap A;H)$ and $u=e^H{\rm Po}(N(0),H)$.
\qed
\end{proof}

We now show that for each $n\in \N_0$ the map $A\mapsto g(n,H,A)$ mapping ${\cal P}(\N_0)\to R(T)$ is an $R(T)$-valued measure (i.e. countably additive set function mapping the empty set to $0$) on $\N_0$ with domain the power set of $\N_0$.

\begin{lemma}\label{cor-2024-1}
For each $n\in \N_0$,
the mapping  $g(n,H,\cdot):\mathcal{P}(\N_0)\to R(T)$ which takes $A\subset \N_0$ to $g(n,H,A)$ is an $R(T)$-valued measure with $g(n,H,\N_0)=0$.
 \end{lemma}

 \begin{proof} For each $n\in \N_0$,
the map  $A\mapsto \nu_A(0)- \nu_A(n)$ is a $\{-u,0,u\}$-valued measure on $\mathcal{P}(\N_0)$.  From (\ref{2024-1-1}), we have
\begin{equation}\label{2024-sept-1}
 (u-u_H)g(n,H,A)=\frac{1}{n}(\nu_A(0)- \nu_A(n)) (u-u_H),
\end{equation}
for $A\subset \N_0$ and $n\in \N_0$.
Thus the map 
$A\mapsto (u-u_H)g(n,H,A)$ is an $R(T)$-valued measure on ${\cal P}(\N_0)$.
Now $\nu_\emptyset(0)=0=\nu_\emptyset(n)$ and 
$\nu_{\N_0}(0)=u=\nu_{\N_0}(n)$ giving $\nu_A(0)-\nu_A(n)=0$ for $A=\emptyset$ and $A=\N_0$. Hence $(u-u_H)g(n,H,\emptyset)=0=(u-u_H)g(n,H,\N_0)$.
 
The map $A\mapsto {\rm Po}(A;H)$ is an $R(T)$-valued measure on ${\cal P}(\N_0)$ and so is the map $A\mapsto{\rm  Po}(A\cap N(n);H)$.
Thus,  from Lemma \ref{lem-2024-1},
 \begin{equation*}
  A\mapsto F(n,H,A)
  \end{equation*}
  is an $R(T)$-valued measure with
 \begin{equation*}
  F(n,H,\N_0)=0
  \end{equation*}
 since 
 ${\rm Po}(\N_0;H)=u$.
 \qed
 \end{proof}

\begin{lemma}\label{j<i}
 For each $i\in\N$ we have that  $Hg(j,H,\{i\})$ is negative decreasing function of $j$ for $j\le i$ with $j\in\N$, and
 \begin{equation}\label{2024-1-12-9}
 Hg(i,H,\{i\})=\frac{-H}{i}{\rm Po}(N(i-1);H).
 \end{equation}
 \end{lemma}

\begin{proof}
 For $j\le i$,  from Lemma \ref{lem-2024-1}, we have that 
\begin{equation}\label{2024-1-12-2}
H g(j,H,\{i\})=-(j-1)!J^{j-1}e^H {\rm Po}(\{i\};H)\, {\rm Po}(N(j-1);H),
\end{equation}
where we have used that $A\cap N(j-1)=\emptyset$, making ${\rm Po}(A\cap N(j-1);H)=0$.
Further, from (\ref{2024-1-12-1}),
\begin{equation}\label{2024-1-12-3}
J^{j-1}e^H {\rm Po}(N(j-1);H)= \sum_{k=0}^{j-1}\frac{J^{j-k-1}}{k!}=\sum_{i=0}^{j-1}\frac{J^{i}}{(j-i-1)!}\ge 0.
\end{equation}
Thus $Hg(j,H,\{i\})$ is negative and
$$j!J^je^H {\rm Po}(N(j);H)=j!J^{i}+
\sum_{i=0}^{j-1}\frac{(j-1)!J^{i}}{(j-i-1)!}
+\sum_{i=0}^{j-1}\frac{i}{j-i}\,\frac{(j-1)!J^{i}}{(j-i-1)!},$$
giving
$$j!J^je^H {\rm Po}(N(j);H)\ge (j-1)!J^{j-1}e^H {\rm Po}(N(j-1);H).$$
Hence proving that $j\mapsto Hg(j,H,\{i\})$ is decreasing.

Finally, taking $i=j$ in (\ref{2024-1-12-2}) gives
$$H g(j,H,\{j\})
=-(j-1)!J^{j-1}e^H\frac{H^{j}e^{-H}}{j!} {\rm Po}(N(j-1);H),$$
 from which (\ref{2024-1-12-9}) follows.
 \qed
 \end{proof}

\begin{lemma}\label{j ge i}
 For fixed $i$ and $j> i$,  we have that the map $j\mapsto Hg(j,H,\{i\})$ is non-negative, decreasing and 
 $$Hg(j,H,\{j-1\})={\rm Po}(\{j,j+1,\dots\};H)=u-{\rm Po}(N(j-1),H).$$
\end{lemma}

\begin{proof}
 From Lemma \ref{lem-2024-1} with $j> i$  we have 
\begin{equation}\label{2024-15-1}
H g(j,H,\{i\})=(j-1)!J^{j-1}e^H {\rm Po}(i;H)(u-{\rm Po}(N(j-1);H)).
\end{equation}
From (\ref{2024-1-12-1}),
\begin{equation}\label{2024-15-3}
e^H (u-{\rm Po}(N(j-1);H))= e^H-\sum_{k=0}^{j-1}\frac{H^k}{k!}.
\end{equation}
Combining (\ref{2024-15-1}) and (\ref{2024-15-3}) gives
\begin{equation}\label{2024-15-4}
H g(j,H,\{i\})={\rm Po}(i;H) \sum_{k=j}^\infty\frac{(j-1)!H^{k-j-1}}{k!}
={\rm Po}(i;H) \sum_{s=1}^\infty\frac{(j-1)!}{(j+s-1)!}H^s\ge 0.
\end{equation}
 Here, for $s\in\N$,
\begin{eqnarray*}
 \frac{(j-1)!}{(j+s-1)!}=\frac{1}{j\cdot(j+1)\cdot\dots \cdot(j+s-1)},
\end{eqnarray*}
 which is decreasing in $j$. 
 
 Finally, from (\ref{2024-15-4}) and (\ref{2024-1-12-1}),
\begin{eqnarray*}
H g(j,H,\{j-1\})={\rm Po}(j-1;H) \sum_{k=j}^\infty\frac{(j-1)!H^{k-j-1}}{k!}
=e^{-H} \sum_{k=j}^\infty\frac{H^{k}}{k!},
\end{eqnarray*}
as claimed.
 \qed
 \end{proof}

\begin{lemma}\label{delta-bound}
Let $A\subset \N_0$ then
$$-(u-e^{-H})\le \Delta(j,H,A)\le u-e^{-H},$$
where
$$\Delta(j,H,A):=H(g(j+1,H,A)-g(j,H,A))$$ for $A\subset \N_0$.
\end{lemma}

\begin{proof}
Using (\ref{2024-1-12-9}) we obtain
$$Hg(j,H,\{j\})=-\frac{e^{-H}}{j} \sum_{k=0}^{j-1}\frac{H^{k+1}}{k!},$$
while from Lemma \ref{j ge i}, 
 $$Hg(j+1,H,\{j\})=e^{-H}\sum_{k=j+1}^\infty \frac{H^k}{k!}.$$
Thus
$$\Delta(j,H,\{j\})=e^{-H}\left[ \sum_{k=j+1}^\infty\frac{H^{k}}{k!}+\frac{1}{j} \sum_{k=0}^{j-1}\frac{H^{k+1}}{k!}\right],$$
where
$$\frac{1}{k!j}\le \frac{1}{(k+1)!},$$
as $k+1\le j$, giving
\begin{equation}\label{2024-1-16-1}
\Delta(j,H,\{j\})\le e^{-H} \sum_{k=1}^\infty\frac{H^{k}}{k!}=u-e^{-H}.
\end{equation}

Lemmas \ref{j ge i} and \ref{j<i} give that
\begin{equation}\label{2024-1-16-5}
\Delta(j,H,\{i\})\le 0,
\end{equation}
for all $i\ne j$.
Now as $g(j,H,\cdot)$ is an $R(T)$-valued measure, so is $\Delta(j,H,\cdot)$, thus from (\ref{2024-1-16-5}) we have
$$\Delta(j,H,A)\le 
\Delta(j,H,\{j\})\le u-e^{-H}.$$

From Lemma \ref{cor-2024-1},
$\Delta(j,H,\N_0)=0$, so $$u-e^{-H}\ge\Delta(j,H,\N_0\setminus A)=-\Delta(j,H,A)$$ giving
$-u+e^{-H}\le \Delta(j,H,A).$
\qed
\end{proof}

As $w$ is a sum of components of $u$, it can be expressed as 
\begin{equation}\label{2024-1-17-3}
w=\sum_{j=0}^n jr_j.
\end{equation}
Here $r_j$ are disjoint components of $u$ summing to $u$.
Taking our lead from functional calculus, see \cite{g-fcal},  as $g(j,H,A)\in R(T)$, we define
$r_j g(w,H,A)= g(j,H,A) r_j$ for each $j$, giving
\begin{equation}\label{2024-1-17-4}
g(w,H,A):=\sum_{j=0}^n r_j g(j,H,A).
\end{equation}
As $r_j$ need not be in $R(T)$, for each $j=0,\dots,n$, we have that  $g(w,H,A)$ need not be in $R(T)$ .

Up until this point we have only need the $R(T)$-module structure of $L^1(T)$ to ensure the existence of products.  Beyond this point in the paper we also use that $L^1(T)$ is an $E_u$-module, see \cite{expectation}, which is a special case of $L^1(T)$ being an $L^\infty(T)$-module, see \cite{KRW}. Further this multipication is commutative and order continuous.

\begin{lemma}\label{lem-2024-17-1}
As $w$ is a finite sum of components of $u$, represented as in (\ref{2024-1-17-3}),  with $A\subset \N_0$ and $g$ as in (\ref{2024-1-17-4}) we have that  
\begin{equation}\label{2024-1-17-6} 
g(w+ku,H,A)=\sum_{j=0}^n r_jg(j+k,H,A),
\end{equation}
for $k\in\N_0$.
Further, if $q$ is a component of $u$, then 
\begin{equation}\label{2024-1-17-8}
qg(w,H,A)=g(qw,H,A).
\end{equation}
\end{lemma}

\begin{proof}
With the notation  set as in the statement of the lemma, 
from (\ref{2024-1-17-3}),
$$w+ku=\sum_{j=0}^n (j+k)r_j.$$
It now follows from (\ref{2024-1-17-3})  and (\ref{2024-1-17-4}) that (\ref{2024-1-17-6}) holds.

From (\ref{2024-1-17-3}),
$$qw=\sum_{j=1}^n j(qr_j)+0(r_0 (u-q)).$$
Hence, from (\ref{2024-1-17-4}),
\begin{eqnarray*}
g(qw,H,A)&=&\sum_{j=1}^n (qr_j)g(j,H,A)+(r_0\vee (u-q))g(0,H,A)\\
&=&\sum_{j=0}^n (qr_j)g(j,H,A)+(u-q)g(0,H,A)\\
&=&qg(w,H,A)+(u-q)g(0,H,A).
\end{eqnarray*}
However, from the definition of $g$, $g(0,H,A)=0$, from which (\ref{2024-1-17-8}) follows.
\qed
\end{proof}

For $w$ a sum of components of $u$, represented as (\ref{2024-1-17-3}), we define the $T$-conditional probability of $w$ taking on components of $u$ times values in $A\subset \N_0$ by
\begin{equation}\label{2024-1-19-1}
\mathbb{P}_T[w\in A]: =\sum_{j\in A} Tr_j\in R(T),
\end{equation}
 and shall refer to it as the  $T$-conditional probability of $w\in A$.

\begin{lemma}\label{rec-2024-1}
For $w$ a sum of components of $u$, represented as in (\ref{2024-1-17-3}),
 and $H=Tw$, we have
 \begin{equation}\label{riesz-recurrence-5}
T(H g(w+u,H,A)-wg(w,H,A))=- {\rm Po}(A;H)+\mathbb{P}_T[w\in A],
\end{equation}
for each $A\subset \N_0$.
\end{lemma}

\begin{proof}
From (\ref{2024-1-1}), with $j$ replaced by $j+1$,  and multiplied by $Hr_j$ we obtain
\begin{equation}\label{riesz-recurrence-2}
Hr_jg(j+1,H,A)=u_Hr_j(jg(j,H,A)- {\rm Po}(A;H)+\nu_A(j)).
\end{equation}
We note that 
$r_j w=jr_j,$ $r_jg(w,H,A)=r_jg(j,H,A)$ and $r_jg(w+u,H,A)=r_jg(j+1,H,A)$ so from (\ref{riesz-recurrence-2}) we get
\begin{equation}\label{riesz-recurrence-3}
Hr_jg(w+u,H,A)=u_Hr_j(w g(j,H,A)+\nu_A(j)-{\rm Po}(A;H)).
\end{equation}
Further to this,  for $j\ge 1$, we have that $0\le r_j\le w$ and thus, from 
Lemma \ref{lem-bands},  $r_j=u_{r_j}\le u_{Tr_j}\le u_{Tw}=u_H,$ giving $r_ju_H=r_j$ and $wu_H=w$.

From the above we have
$$u-r_0=\sum_{j=1}^n r_j\le u_H,$$
so for $j=0$ we have that
$(u-r_0)u_H=u-r_0$ which can be rearranged to give
$u-u_H=r_0(u-u_H)$.
Thus
 $r_0=r_0u_H+(u-u_H)r_0=r_0u_H+(u-u_H)$.
 Hence
 $$\sum_{j\in A}r_j=(u-u_H)\nu_A(0)+u_H\sum_{j\in A}r_j.$$

So summing (\ref{riesz-recurrence-2}) over $j=0,\dots, n$, gives
\begin{equation}\label{riesz-recurrence-3-1}
Hg(w+u,H,A)=w g(w,H,A)+\sum_{j\in A}r_j-(u-u_H)\nu_A(0)-u_H{\rm Po}(A;H).
\end{equation}

By (\ref{dagger}),
$(u-u_H){\rm Po}(A;H)=\nu_A(0)(u-u_H)$,
so from (\ref{riesz-recurrence-3-1}) we get
\begin{equation}\label{riesz-recurrence-3-3}
Hg(w+u,H,A)=w g(w,H,A)+\sum_{j\in A}r_j-{\rm Po}(A;H).
\end{equation}
Applying $T$ to (\ref{riesz-recurrence-3-3}) and using (\ref{2024-1-19-1})
 yields the result of the lemma.
\qed
\end{proof}

\begin{lemma}\label{lem-2024-indep}
Let $q_1,\dots,q_n$ be $T$-conditionally independent components of $u$, $A\subset \N_0$, 
$\displaystyle{w=\sum_{j=0}^n q_j}$, $\displaystyle{w_i=\sum_{j\ne i} q_j}$, $H=Tw$ and $h_j=Tq_j$ for each $j=0,\dots,n$, then we have the $T$-conditional independence relation 
\begin{equation}\label{t-con-indep-2024}
T(q_ig(w_i+ku,H,A))=(Tq_i)Tg(w_i+ku,H,A).
\end{equation}
and
\begin{equation}\label{indep}
T(q_i g(w+ku,H,A))=h_i Tg(w_i+(k+1)u,H,A),
\end{equation}
for each $k\in \N_0$ and $i\in\{0,\dots,n\}$.
\end{lemma}

\begin{proof}
As $w=w_i+q_i$, $q_iq_i=q_i$ and $q_iu=q_i$ we have
\begin{equation}\label{eq-main-2}
q_i (w+ku)=q_i(w_i+q_i+ku)=q_iw_i+(k+1)q_i=q_i(w_i+(k+1)u),
\end{equation}
for $k\in \N_0$.
Combining (\ref{eq-main-2}) with (\ref{2024-1-17-8}) gives
$$q_ig(w+ku,H,A)=g(q_i(w+ku),H,A)=g(q_i(w_i+(k+1)u),H,A)=q_ig(w_i+(k+1)u,H,A),$$
and thus
\begin{equation}\label{eq-main-3}
T(q_ig(w+ku,H,A))=T(q_ig(w_i+(k+1)u,H,A)).
 \end{equation}

Now $w_i$ can be represented as  
 $$w_i=\sum_{j=0}^n jr_j^i,$$ 
 where $r_j^i$ for $j=0,\dots,n$ are disjoint components of $u$ summing to $u$.
 Applying Lemma \ref{lem-2024-17-1} to $w_i$ we have
\begin{equation}\label{2024-1-17-6-h} 
g(w_i+ku,H,A)=\sum_{j=0}^n r^i_jg(j+k,H,A),
\end{equation}
for $k\in\N_0$.
Using that $g(j+k,H,A)\in R(T)$, the averaging property of $T$ and (\ref{2024-1-17-6-h}) we obtain
\begin{equation}\label{eq-inped-17}
T(q_ig(w_i+ku,H,A))=\sum_{j=0}^n T(q_ir^i_j)g(j+k,H,A),
 \end{equation}
Here $\{r^i_0,\dots,r^i_n\}$ are in the order closed Riesz subspace of $E$ generated by $R(T)$ and $q_j$, $j\ne i$.
By the $T$ conditional independence of $\{q_1,\dots,q_n\}$ we have that each $r_j^i$ is conditional independent of $q_i$ and thus
\begin{equation}\label{2024-lem-indep-5}
T(q_ir^i_j)=(Tq_i)(Tr^i_j),
\end{equation}
for each $i$ and $j$.
Combining (\ref{eq-inped-17}), (\ref{2024-lem-indep-5}) and the averaging property of $T$ we get
\begin{equation*}
T(q_ig(w_i+ku,H,A))=(Tq_i)\sum_{j=0}^n T(r^i_j)g(j+k,H,A)=(Tq_i)T\left(\sum_{j=0}^n r^i_jg(j+k,H,A)\right),
\end{equation*}
giving (\ref{t-con-indep-2024}).
Combining (\ref{t-con-indep-2024}) and (\ref{eq-main-3}) gives (\ref{indep}).
\qed
\end{proof}

We are now in a position to prove our first main theorem.

\begin{theorem}[Finite sum law of small numbers]\label{thm-main}
Let $q_1,\dots,q_n$ be $T$-conditionally independent components of $u$, $A\subset \N_0$, $\displaystyle{w=\sum_{j=0}^n q_j}$, $H=Tw$ and $h_j=Tq_j$ for each $j=0,\dots,n$, then 
\begin{equation}\label{eq-main-1}
|\mathbb{P}_T[w\in A]- {\rm Po}(A;H)|\le \sup_{i=0,\dots,n} h_i.
\end{equation}
\end{theorem}

\begin{proof}
From Lemma \ref{lem-2024-indep},
$$T(q_i\cdot g(w+ku,H,A))=h_i\cdot Tg(w_i+(k+1)u,H,A),$$
for each $k\in \N_0$ and $i\in\{0,\dots,n\}$.
Summing the above equation from $i=0$ to $i=n$ gives
\begin{equation}\label{main-2024-11}
T(w\cdot g(w+ku,H,A))=\sum_{i=0}^n h_i\cdot Tg(w_i+(k+1)u,H,A).
\end{equation}
Lemma \ref{rec-2024-1} together with (\ref{main-2024-11}), for $k=0$, gives
\begin{equation}\label{2024-bound-1}
\mathbb{P}_T[w\in A]- {\rm Po}(A;H)=\sum_{i=0}^n h_iT( g(w+u,H,A)-g(w_i+u,H,A)),
\end{equation}
for each $A\subset \N_0$.

 Since $q_i(w+u)=q_i(w_i+2u)$ and $(u-q_i)(w+u)=(u-q_i)(w_i+u),$ by Lemma \ref{lem-2024-17-1}, we have
 \begin{eqnarray*}
g(w+u,H,A)&=&q_ig(w+u,H,A)+(u-q_i)g(w+u,H,A)\\
&=&g(q_i(w+u),H,A)+g((u-q_i)(w+u),H,A)\\
&=&g(q_i(w_i+2u),H,A)+g((u-q_i)(w_i+u),H,A)\\
&=&q_ig(w_i+2u,H,A)+(u-q_i)g(w_i+u,H,A)\\
&=&q_i(g(w_i+2u,H,A)-g(w_i+u,H,A))+g(w_i+u,H,A)).
\end{eqnarray*}
Now applying $T$ to the above and using (\ref{t-con-indep-2024}) with $k=1$ and $k=2$, we obtain
 \begin{equation}\label{eq-main-2024-12-T}
Tg(w+u,H,A)=h_iTg(w_i+2u,H,A)-h_iTg(w_i+u,H,A)+Tg(w_i+u,H,A)).
\end{equation}
But
\begin{equation}\label{2025-april-2}
\sum_{i=0}^n h_i^2T|g(w_i+2u,H,A)-g(w_i+u,H,A)|\le
H \sup_{j=0,\dots,n}h_j T|g(w_i+2u,H,A)-g(w_i+u,H,A)|.
\end{equation}

Combining (\ref{2024-1-17-6}) with (\ref{2024-1-17-6-h}) yields
$$H|g(w_i+2u,H,A)-g(w_i+u,H,A)|\le \sum_{j=0}^n r^i_j H|g(j+2,H,A)-g(j+1,H,A)|
= \sum_{j=0}^n r^i_j |\Delta(j+1,H,A)|.$$
So by Corollary \ref{cor-2024-1} we have
\begin{equation}\label{2025-april-1}
H|g(w_i+2u,H,A)-g(w_i+u,H,A)|\le\sum_{j=0}^n r^i_j (u-e^{-H})=u-e^{-H}.
\end{equation}
Combining (\ref{2024-bound-1}, (\ref{eq-main-2024-12-T}) and  (\ref{2025-april-1}) gives
\begin{equation}\label{2024-bound-2}
|\mathbb{P}_T[w\in A]- {\rm Po}(A;H)|=\sup_{j=0,\dots,n}h_j(u-e^{-H}),
\end{equation}
from which the theorem follows.
\qed
\end{proof}

Our last result extends the above bound to the case of infinite sequences.
For this we require the following lemma on convergence.
\begin{lemma}\label{lem-conv-dist}
Let  
$(q_j)$ be a sequence of components of $u$, with  
$$s:=\sum_{j=0}^\infty q_j\in L^1(T).$$
Denote
$$s_n:=\sum_{j=0}^n q_j,$$
 then
\begin{equation}\label{conv-dist}
\mathbb{P}_T[s_n\in A]
\to 
\mathbb{P}_T[s\in A],
\end{equation}
in order as $n\to\infty$,
 for each $A\subset \N_0$.
\end{lemma}

\begin{proof}
We being by expressing $s$ and $s_n$ in terms of disjoint components of $u$ as
\begin{eqnarray*}
 s&=&\sum_{j=0}^\infty j r_j,  \mbox{ where } u=\sum_{j=0}^\infty  r_j, \\
 s_n&=&\sum_{j=0}^n j r_j^n  \mbox{ where } u=\sum_{j=0}^n  r_j^n,
\end{eqnarray*}
and $r_j, r_j^n$ are components of $u$.  Here $r_j\cdot r_k=\delta_{jk}r_j$ and
$r_j^n\cdot r_k^n=\delta_{jk}r_j^n$.
If we set
$$u_n:=\sum_{j=0}^n r_j\cdot r_j^n,$$
then $u_n\cdot(s-s_n)=0$ and $(u-u_n)\cdot (s-s_n)\ge u-u_n$, giving that 
\begin{equation}\label{2025-1}
P_{((s-s_n)-\frac{1}{2} u)^+}u=u-u_n.
\end{equation}
Now $s_n\uparrow s$ as $n\to\infty$ so $|s_n -s|\to 0$ in order as $n\to \infty$, which, together with the order continuity of $T$ gives $T|s_n-s|\to 0$ in order as $n\to\infty$, i.e. in the notation of \cite{AKRW}, $s_n\to s$ T-strongly.  Now, by \cite[Lemma 5.3]{AKRW}, it follows that $s_n\to s$ in T-conditional probability, that is
\begin{equation}\label{2025-2}
TP_{((s-s_n)-\epsilon u)^+}u\to 0,
\end{equation}
 in order as $n\to\infty$ for each $\epsilon>0$.
In particular taking $\epsilon=1/2$ in (\ref{2025-2}) and combining it with (\ref{2025-1}), we have that
$T(u-u_n)\to 0$ in order as $n\to\infty$.

We now consider the T-conditional probabilities of interest. As $u_nr_j=u_nr_j^n$ we have
\begin{eqnarray*}
|\mathbb{P}_T[s_n\in A]- \mathbb{P}_T[s\in A]|
&=&\left|\sum_{j\in A\cap\{0,\dots,n\}}Tr_j^n-\sum_{j\in A} Tr_j\right|\\
&\le &T\left|\sum_{j\in A\cap\{0,\dots,n\}}(u_nr_j^n+(u-u_n)r_j^n)-\sum_{j\in A} (u_nr_j+(u-u_n)r_j)\right|\\
&\le&T\left((u-u_n)\left|\sum_{j\in A\cap\{0,\dots,n\}}r_j^n-\sum_{j\in A} r_j\right|\right)\\
&\le&T(u-u_n)\to 0
\end{eqnarray*}
as $n\to\infty$.
\qed
\end{proof}

\begin{theorem}[Infinite sum law of small numbers]\label{lsn-inf}
Let  
$(q_j)$ be a $T$-conditionally independent sequence of components of $u$, with  
$$s:=\sum_{j=0}^\infty q_j\in L^1(T),$$
 then,  for all $A\subset \N_0$,
\begin{equation}\label{rs-lsn}
|\mathbb{P}_T[s\in A]- {\rm Po}(A;Ts)|\le \sup_{i\in\N_0} h_i,
\end{equation}
where $h_i=Tq_i$.
\end{theorem}

\begin{proof}
Let
$$s_n:=\sum_{j=0}^n q_j\in L^1(T),$$
$H_n:=Ts_n$ and $h_j=Tq_j$.  Then Theorem \ref{thm-main} gives
\begin{equation}\label{eq-cor-1}
|\mathbb{P}_T[s_n\in A]- {\rm Po}(A;H_n)|\le \sup_{i=0,\dots,n} h_i\le \sup_{i\in\N_0} h_i.
\end{equation}
We note that the supremum on the right of (\ref{eq-cor-1}) exists due to the Dedekind completeness of $L^1(T)$, as $0\le h_j\le u$ for all $j$.

By Lemma \ref{lem-conv-dist}
$$\mathbb{P}_T[s_n\in A]\to \mathbb{P}_T[s\in A],$$
in order as $n\to \infty$.

Further $s_n\uparrow s$, so by the order continuity of $T$, $H_n=Ts_n\to Ts$ in order as $n\to\infty$ and 
\begin{equation}\label{2024-1-12-30}
{\rm Po}(A;H_n)=e^{-H_n}\sum_{0\ne k\in A}\frac{H_n^k}{k!}+\chi_A(0)(I-P_{H_n})u,
\end{equation}
for each $A\subset \N_0$ and $H\in R(T)$.
Here, $\chi_A(0)(I-P_{H_n})u\downarrow \chi_A(0)(I-P_{Ts})u,$ so
$\chi_A(0)(I-P_{H_n})u\to \chi_A(0)(I-P_{Ts})u$ in order as $n\to \infty$.
The exponential map $v\mapsto e^{-v}$ is order continuous for $v\in E_+$ and finally
$$\sum_{0\ne k\in A}\frac{H_n^k}{k!}\uparrow \sum_{0\ne k\in A}\frac{(Ts)^k}{k!}\le e^{Ts}\in R(T)\subset L^1(T),$$
so by the Dedekind completeness of $L^1(T)$,
$$\sum_{0\ne k\in A}\frac{H_n^k}{k!}\to \sum_{0\ne k\in A}\frac{(Ts)^k}{k!},$$
as $n\to\infty$. Hence, from the order continuity of the multiplication, we can take the order limit as $n\to\infty$ in (\ref{eq-cor-1}) to give (\ref{rs-lsn}). 
\qed
\end{proof}


\section{Application}

We now apply the above result to a general probability space $(\Omega,{\cal F},\mu)$  and $(B_i)$ be a sequence of events in ${\cal F}$, which are conditionally independent with respect to $\Sigma$, a sub-$\sigma$ algebra of ${\cal F}$.  For $B\in {\cal F}$ we have that 
$\P[B|\Sigma]= \E[\chi_B|\Sigma]$.
Let $$s=\sum_{k=0}^\gamma \chi_{B_k},$$
where $\gamma\in \N\cup\{\infty\}$.
If $$\lim_{n\uparrow \gamma} \sum_{k=0}^n \E[\chi_{B_k}|\Sigma]$$
exists a.e. pointwise, then this limit is denoted $\E[s|\Sigma]$ and
 from Theorems \ref{thm-main} and \ref{lsn-inf}, for  $A\subset \N_0$, we have that 
 \begin{equation}\label{final-2026}
 |\mathbb{P}[s\in A|\Sigma]-{\rm Po}(A;\E[s|\Sigma])|\le \sup_{k\in\{0,\dots,\gamma\}}
\mathbb{P}[B_k|\Sigma].
\end{equation}

We now illustrate this application with some computational examples. In the first simple example, both independence and conditional independence are available, but the outcomes of the classical theory with independence is compare to the outcome using conditional independence.
In the second simple example the random variables are only conditionally independent, hence making only the extended theory presented here applicable.  In the final example, not quite so simple, we give a sequence of
random variables which is conditionally independent but not independent and we apply to it our conditional theory.  Further to this, we highlight that the error bounds resulting from our conditional Stein-Chen method is a functional bound, yielding regions of good approximation and regions of poor approximation. By contrast the classical Stein-Chen method yields a constant  numeric bound applicable uniformly over the entire domain, which is uniformly good or uniformly bad.

{\bf Example 1.}
Let $\Omega=\{1,2,3,4\}$, ${\cal F}={\cal P}(\Omega)$ and $\mathbb{P}$ the measure on ${\cal F}$ generated by
$\mathbb{P}(\{j\})=\left\{\begin{array}{ll}3/8,& j=1,4\\ 1/8,& j=2,3\end{array}\right.$. Let $\Sigma=\{\emptyset, \{1,2\}, \{3,4\},\Omega\}$. Set $B_1=\{2,3\}$ and $B_2=\{3,4\}$ and let $s=\chi_{B_1}+\chi_{B_2}$.

Here $\mathbb{P}(B_1)=1/4$ and $\mathbb{P}(B_2)=1/2$, but $\mathbb{P}(B_1\cap B_2)=1/8=\mathbb{P}(B_1)\mathbb{P}(B_2)$, so $B_1$ and $B_2$ are independent. Further $\E[s]=3/4$. So the results of Stein and Chen give that 
$A\subset \N_0$, we have that 
$$|\mathbb{P}[s\in A]-{\rm Po}(A;3/4)|\le  1/2.$$

However, 
$\mathbb{P}[B_1|\Sigma](j)=\left\{\begin{array}{ll}1/4,& j=1,2\\ 1/4,& j=3,4\end{array}\right.$ and
$\mathbb{P}[B_2|\Sigma](j)=\left\{\begin{array}{ll}0,& j=1,2\\ 1,& j=3,4\end{array}\right.$ and thus
$$\mathbb{P}[B_1|\Sigma]\mathbb{P}[B_2|\Sigma](j)=\left\{\begin{array}{ll}0,& j=1,2\\ 1/4,& j=3,4\end{array}\right.=\mathbb{P}[B_1\cap B_2|\Sigma](j)$$
for $j=1,2,3,4$. Hence $B_1$ and $B_2$ are conditionally independent 
 with respect to $\Sigma$ and as such (\ref{final-2026}) is applicable. Here
$\mathbb{P}[B_1|\Sigma]\vee \mathbb{P}[B_2|\Sigma](j)=\left\{\begin{array}{ll}1/4,& j=1,2\\ 1,& j=3,4\end{array}\right.$
and $\mathbb{E}[s|\Sigma](j)=\left\{\begin{array}{ll}1/4,& j=1,2\\ 5/4,& j=3,4\end{array}\right.$ and 
(\ref{final-2026}) gives
$$ |\mathbb{P}[s\in A|\Sigma](j)-{\rm Po}(A;1/4)|\le 1/4$$
for $j=1,2$, and
$$ |\mathbb{P}[s\in A|\Sigma](j)-{\rm Po}(A;5/4)|\le 1$$
for $j=3,4$.
Note here the distinctly better bound obtain for $j=1,2$.

{\bf Example 2.}
In Example 1. reset the measure to be that given by
$\mathbb{P}(\{j\})=\left\{\begin{array}{ll}3/8,& j=1,3\\ 1/8,& j=2,4\end{array}\right.$.
Now
 $\mathbb{P}(B_1)=1/2=\mathbb{P}(B_2)$, but $\mathbb{P}(B_1\cap B_2)=3/8\ne \mathbb{P}(B_1)\mathbb{P}(B_2)$, so 
 $B_1$ and $B_2$ are no longer independent.
 However 
 $\mathbb{P}[B_1|\Sigma](j)=\left\{\begin{array}{ll}1/4,& j=1,2\\ 3/4,& j=3,4\end{array}\right.$ and
$\mathbb{P}[B_2|\Sigma](j)=\left\{\begin{array}{ll}0,& j=1,2\\ 1,& j=3,4\end{array}\right.$ and thus
$$\mathbb{P}[B_1|\Sigma]\mathbb{P}[B_2|\Sigma](j)=\left\{\begin{array}{ll}0,& j=1,2\\ 3/4,& j=3,4\end{array}\right.=\mathbb{P}[B_1\cap B_2|\Sigma](j)$$
for $j=1,2,3,4$. Hence $B_1$ and $B_2$
 are still conditionally independent with respect to $\Sigma$, hence (\ref{final-2026}) is applicable, but the standard theory of Stein and Chen in not. Here (\ref{final-2026}) yields
$$ |\mathbb{P}[s\in A|\Sigma](j)-{\rm Po}(A;1/4)|\le 1/4$$
for $j=1,2$, and
$$ |\mathbb{P}[s\in A|\Sigma](j)-{\rm Po}(A;7/4)|\le 1$$
for $j=3,4$.

It should be note in both of these example the occurence of $B_1$ is rare, but $B_2$ is only conditionally rare on $\{1,2\}$ but not on $\{3,4\}$, hence the usefulness of the bounds for $j=1,2$ and the uselessness for $j=3,4$.

{\bf Example 3.}
Let $\Omega=\N$ and ${\cal F}={\cal P}(\N)$.  Set 
$$\mathbb{P}(\{n\})=\left\{\begin{array}{ll}\frac{2}{3\cdot 2^k}, &n=2k-1, k\in\N,\\
\frac{1}{2^{2k}},&n=2k, k\in\N,
\end{array}\right.$$
 with countably additive extension to a probability measure on ${\cal P}(\N)$.   

Let $\Sigma$ be the sub-$\sigma$ algebra of ${\cal P}(\N)$ generated by the collection 
$\{\{2n-1, 2n\}\,|\,n\in\N\}$.
Thus for $f:\N\to \R$ we have
$$\E[f|\Sigma](n)
=\frac{
  \int_{\{2k-1,2k\}} f\,d\mathbb{P}}{\mathbb{P}(\{2k-1,2k\})}
=\frac{3 f(2k)+2^{k+1}f(2k-1)}{3+2^{k+1}} \mbox{ for } n\in\{2k-1,2k\},$$
with $n,k\in\N$.

Let $B_{2k-1}=\{ 4k-3,4k-2, 4k+1, 4k+2, 4k+5, 4k+6, \dots\}$ and $B_{2k}=\{4k\}$ for $k\in\N$.
It can be verified that $(B_j)$ is a $\P[\cdot|\Sigma]$ conditionally independent sequence.
Setting 
$p_j=\chi_{B_j}$ we have 
$$\E[p_{2k}|\Sigma](n)=\left\{\begin{array}{ll}\frac{3}{3+2^{2k+1}}, &n\in\{4k-1,4k\}\\
0,&\mbox{otherwise}\end{array}\right.,\, k,n\in\N,$$
and
\begin{align*}\E[p_{2k-1}|\Sigma](n)
&=p_{2k-1},\quad k\in \N.
\end{align*}

Thus
$$\sup_{j\in\N}
\mathbb{E}[p_j|\Sigma](n)=\left\{\begin{array}{ll}\frac{3}{3+2^{2k+1}}, &n\in\{4k-1,4k\}\\
1,&\mbox{otherwise}\end{array}\right.,\,n,k\in\N.$$
Now
$$s(n)=\sum_{j=1}^\infty p_j(n)=
\left\{\begin{array}{ll}
0, &n=4k-1,\\
1, &n=4k,\\
k, &n\in \{4k-3,4k-2\}
\end{array}\right. \,k\in\N.$$
Thus
$$\E[s|\Sigma](n)
=\left\{\begin{array}{ll}
\frac{3}{3+2^{2k+1}}, &n\in\{4k-1,4k\}\\
k, &n\in \{4k-3,4k-2\}
\end{array}\right. \,k\in\N,$$
giving
\begin{eqnarray*}{\rm Po}(j;\E[s|\Sigma])(n)&=&\frac{(\E[s|\Sigma](n))^j}{j!}e^{-\E[s|\Sigma](n)}\\
&=&\left\{\begin{array}{ll}
\frac{\left(\frac{3}{3+2^{2k+1}}\right)^j}{j!}{\rm exp}\left(-\frac{3}{3+2^{2k+1}} \right), &n\in\{4k-1,4k\}\\
\frac{k^j}{j!}e^{-k}, &n\in \{4k-3,4k-2\}
\end{array}\right. \,k\in\N,
\end{eqnarray*}

Now for $A\subset \N_0$,  we have
\begin{eqnarray*}|\mathbb{P}[s\in A|\Sigma](n)-{\rm Po}(A;\E[s|\Sigma])(n)|
&\le& 
\sup_{k\in\N}
\mathbb{P}[B_k|\Sigma](n)\\ &=&\left\{\begin{array}{ll}\frac{3}{3+2^{2k+1}}, &n\in\{4k-1,4k\}\\
1,&n\in\{4k-2,4k-3\}\end{array}\right.k\in \N
\end{eqnarray*}
a bound on the approximation which is useful for $n\in \{3,4,7,8,11,12,\dots\}$ and useless for $n\in \{1,2,5,6,9,10,\dots\}$.

 \end{document}